\documentclass{amsart}
\usepackage{amsfonts,amsmath,amsthm,amssymb,}
\usepackage[latin1]{inputenc}
\usepackage{graphicx}
\usepackage{url}

\usepackage{fancyhdr}

\usepackage{enumerate}
\usepackage[breaklinks]{hyperref}
\usepackage{verbatim}
\usepackage[mathscr]{euscript}
\usepackage{enumerate,multicol,array}

\title[The model theory of modules of a $C^*$-algebra]{The model theory of modules of a $C^*$-algebra}

\author{Camilo Argoty}
\address{Camilo Argoty
\\ Departamento de Matem\'aticas
\\ Universidad Sergio Arboleda
	\\ Bogot\'a, Colombia}

\date{}

\thanks{The author is very thankful to Alexander Berenstein for his help in reading and correcting this work. In the same way, the author wants to thank Ita\"i Ben Yaacov for discussing ideas in particular cases. Seemingly, the author wants to thank C. Ward Henson for his advice and help, and for sharing with the autor the ideas underlying Section \ref{Htheory}. Finally, the author wants to thank Andr\'es Villaveces for his interest and for sharing ideas on how to extend the results of this paper.}

\makeatletter
\def\newrefformat#1#2{%
  \@namedef{pr@#1}##1{#2}}
\def\prettyref#1{\@prettyref#1:}
\def\@prettyref#1:#2:{%
  \expandafter\ifx\csname pr@#1\endcsname\relax%
    \PackageWarning{prettyref}{Reference format #1\space undefined}%
    \ref{#1:#2}%
  \else%
    \csname pr@#1\endcsname{#1:#2}%
  \fi%
}
\makeatother

\newrefformat{eq}{\textup{(\ref{#1})}}
\newrefformat{cha}{Chapter \ref{#1}}
\newrefformat{sec}{Section \ref{#1}}
\newrefformat{tab}{Table \ref{#1} on page \pageref{#1}}
\newrefformat{fig}{Figure \ref{#1} on page \pageref{#1}}

\makeatletter
\def\indsym#1#2{%
  \setbox0=\hbox{$\m@th#1x$}%
  \kern\wd0%
  \hbox to 0pt{\hss$\m@th#1\mid$\hbox to 0pt{$\m@th#1^{#2}$}\hss}%
  \lower.9\ht0\hbox to 0pt{\hss$\m@th#1\smile$\hss}%
  \kern\wd0} 
\def\nindsym#1#2{%
  \setbox0=\hbox{$\m@th#1x$}%
  \kern\wd0%
  \hbox to 0pt{\mathchardef\nn="3236\hss$\m@th#1\nn$\kern1.4\wd0\hss}
  \hbox to 0pt{\hss$\m@th#1\mid$\hbox to 0pt{$\m@th#1^{#2}$}\hss}%
  \lower.9\ht0\hbox to 0pt{\hss$\m@th#1\smile$\hss}%
  \kern\wd0}

\makeatother

 \def\ba{\bar a}  
 \def\bv{\bar v}

\def\p{\pi}\def\f{\phi}\def\a{\alpha}

\def\ben{\begin{enumerate}}\def\een{\end{enumerate}}
\def\bdc{\begin{description}}\def\edc{\end{description}}
\def\bitm{\begin{itemize}}\def\eitm{\end{itemize}}
\def\bdf{\begin{defin}}\def\edf{\end{defin}}
\def\bth{\begin{theo}}\def\eth{\end{theo}}
\def\bfc{\begin{fact}}\def\efc{\end{fact}}
\def\bco{\begin{coro}}\def\eco{\end{coro}}
\def\brm{\begin{rem}}\def\erm{\end{rem}}
\def\blm{\begin{lemma}}\def\elm{\end{lemma}}
\def\bnt{\begin{nota}}\def\ent{\end{nota}}
\def\bex{\begin{exe}}\def\eex{\end{exe}}
\def\bpf{\begin{proof}}\def\epf{\end{proof}}
\def\bas{\begin{assum}}\def\eas{\end{assum}}
\def\beq{\begin{equation}}\def\eeq{\end{equation}}
\def\becl{\begin{cla}}\def\ecl{\end{cla}}

  \def\d{\delta}

 \def\r{\rho}
 \def\k{\kappa}  \def\a{\alpha}
 \def\L{L}   \def\l{\lambda}  \def\L{\Lambda}
   \def\f{\phi}

\def\A{\mathcal{A}}
\def\FF{\mathcal{F}}

\def\O+{\oplus}

\def\dotminus{-^{\! \! \! \!  \cdot}}

\newtheorem{theo}{Theorem}[section]

\newtheorem{coro}[theo]{Corollary}
\newtheorem{lemma}[theo]{Lemma}

\theoremstyle{definition}
\newtheorem{defin}[theo]{Definition}
\newtheorem{fact}[theo]{Fact}
\theoremstyle{remark}
\newtheorem{exe}[theo]{Example}
\newtheorem{rem}[theo]{Remark}
\newtheorem{assum}[theo]{Assumption}
\newtheorem{nota}[theo]{Notation}
\newtheorem*{cla}{Claim}

\newcommand{\N}{\ensuremath{\mathbb{N}}}
\newcommand{\Z}{\ensuremath{\mathbb{Z}}}

\newcommand{\R}{\ensuremath{\mathbb{R}}}
\newcommand{\C}{\ensuremath{\mathbb{C}}}

\newcommand{\MM}{\ensuremath{\mathcal{M}}}
\newcommand{\Aut}{\ensuremath{\text{Aut}}}

\newcommand{\sfrac}[2]{\hbox{$\frac{#1}{#2}$}}
\newcommand{\half}[1][1]{\sfrac{#1}{2}}

\begin{document}

\begin{abstract}
We study the theory of a Hilbert space $H$ as a module for a unital $C^*$-algebra $\A$ from the point of view of continuous logic. We give an explicit axiomatization for this theory and describe the structure of all the representations which are elementary equivalent to it.  We show that for every $v\in H$ the type $tp(v/\emptyset)$ is in correspondence with the positive linear functional over $\A$ defined by $v$ and has quantifier elimination as well. Finally, we characterize the model companion of the incomplete theory of all non-degenerate representations of $\A$.
\end{abstract}

\maketitle

\section{introduction}

Let $\A$ be a (unital) algebra and let $\p:\A\to B(H)$ be a $C^*$-algebra nondegenerate isometric homomorphism, where $B(H)$ is the algebra of bounded operators over a Hilbert space $H$. The goal of this paper is to study $H$ as a metric structure expanded by $\A$ from the point of view of continuous logic (see \cite{BBHU} and \cite{Ben3}). In order to describe the structure of $H$ as a module for $\A$, we include a symbol $\dot{a}$ in the language of the Hilbert space structure whose interpretation in $H$ will be $\p(a)$ for every $a$ in the unit ball of $\A$. Following \cite{Ben3}, we study the theory of $H$ as a metric structure of only one sort:
\[(Ball_1(H),0,-,i,\half[x+y],\|\cdot\|,(\p(a))_{a\in Ball_1(\A)})\] where $Ball_1(H)$  and $Ball_1(\A)$ are the corresponding unit balls in $H$ and $\A$ respectively; $0$ is the zero vector in $H$; $-: Ball_1(H)\to Ball_1(H)$ is the function that to any vector $v\in Ball_1(H)$ assigns the vector $-v$; $i: Ball_1(H)\to Ball_1(H)$ is the function that to any vector $v\in Ball_1(H)$ assigns the vector $iv$ where $i^2=-1$; $\half[x+y]: Ball_1(H)\times Ball_1(H)\to Ball_1(H)$ is the  function that to a couple of vectors $v$, $w\in Ball_1(H)$ assigns the vector $\half[v+w]$; $\|\cdot\|: Ball_1(H) \to [0, 1]$ is the norm function; $\A$ is an unital $C^*$-algebra; $\p:\A\to B(H)$ is a $C^*$-algebra isometric homomorphism. The metric is given by $d(v,w) =\|\half[v-w]\|$. Briefly, the structure will be refered to as $(H,\p)$.

It is worthy noting that with this language, we can define the inner product taking into account that for every $v$, $w\in Ball_1(H)$, \[\langle v\ |\ w\rangle=\|\half[v+w]\|^2-\|\half[v-w]\|^2+i(\|\sfrac{v+iw}{2}\|^2-\|\half[v-iw]\|^2)\] Because of this reason, we will make free use of the inner product as if it were included in the language. In most arguments, we will forget this formal point of view, and will treat $H$ directly. To know more about the continuous logic point of view of Banach spaces please see \cite{Ben3}, Section $2$.

The model theory of expansions of Hilbert spaces by bounded operators has been studied extensively. Henson and Iovino in \cite{Io}, observed that the theory of a Hilbert space expanded with a family of bounded operators is stable. The author and Berenstein (\cite{ArBer}) studied the theory of the structure $(H,+,0,\langle|\rangle,U)$ where $U$ is a unitary operator with countable spectrum and characterized prime models and orthogonality of types. The author and Ben Yaacov (\cite{ArBen}) studied the more general case of a Hilbert space expanded by a normal operator $N$. Most results in this paper are generalizations of results present in \cite{ArBer,ArBen}. In \cite{BUZ} Ben Yaacov, Usvyatsov and Zadka characterized the unitary operators corresponding to generic automorphisms of a Hilbert space as those unitary transformations whose spectrum is $S^1$ and gave the key ideas used in this paper to characterize domination and orthogonality of types.

A work related to this one is the one of Farah, Hart and Sherman who recently have showed that the theory of a $C^*$-algebra is not stable (See \cite{FaHaShe1} and \cite{FaHaShe2}). These papers and Farah's work point out one phenomenon: $C^*$-algebras have complicated model theoretical structure but their representations are very well behaved. This is similar to the case of the integers $\Z$: The theory $Th(\Z)$ is quite difficult from the model theoretic point of view, but some of its representations like torsion free abelian groups are very well behaved.

This paper is divided as follows: In Section \ref{sectiontheoryofHpi}, we give an explicit axiomatization of $Th(H,\p)$; many of the results from this section were proved by C.W. Henson but did not appear in print. In Section \ref{SectionModelsOfTpi} we give a description of the models of $Th(H,\p)$ and build the monster model for the theory. In Section \ref{sectiontypesquantifierelimination}, we characterize the types over the emtpy set as positive linear functionals on $\A$ and prove quantifier elimination. Finally, in section \ref{SectionModelCompanionForTA}, we build a model companion for the incomplete theory of all non-degenerate representations of a $C^*$-algebra $\A$.

We will assume the reader is familiar with basic concepts of spectral theory, for example the material found in \cite{Pe}. We will recall technical results from \cite{Pe} as we need them.

\section{the theory $IHS_{\A,\p}$}\label{sectiontheoryofHpi}

In this section we provide an explicit axiomatization of $Th(H,\p)$. The main tool here is Theorem \ref{Voiculescu} which is mainly a consecuence of Voiculescu's theorem (see \cite{Da}). This Theorem states that two separable representations $(H_1,\p_1)$ and $(H_2,\p_2)$ of a separable $C^*$-algebra $\A$ are approximately unitarily equivalent if and only if for every $a\in\A$, rank$(\p_1(a))=$rank$(\p_2(a))$. This last statement can be expressed in continuous first order logic and is the first step to build the axiomatization we mentioned above. Lemma \ref{HensonsLemma}, Theorem \ref{equiunel} and Corollary \ref{Htheory} are remarks and unpublished results from C. Ward Henson.

\begin{defin}
Let $\A$ be a $C^*$-algebra. A \textit{representation} is an isometric algebra homomorphism $\p:\A\to B(H)$ such that for all $a\in\A$, $\p(a^*)=(\p(a))^*$. In this case $H$ is called an $\A$-\textit{module}. A Hilbert subspace $H^\prime\subseteq H$ is called an $\A$-\textit{submodule} or a \textit{reducing $\A$-subspace} of $H$ if $H^\prime$ is closed under $\p$. $H$ is called $\A$-\textit{irreducible} or $\A$-\textit{minimal} if $H$ has no proper non trivial $\A$-submodules. The set of representations of an algebra $\A$ on $B(H)$ is denoted $rep(\A,B(H))$.
\end{defin}

\bdf
Let $(H,\p)$ be a representation of a $C^*$-algebra $\A$. $(H,\p)$ is called \textit{non-degenerate} if for every nonzero vector $v\in H$, there exists $a\in\A$ such that $\p(a)v\neq 0$.
\edf

\bfc[Remark 2.2.4 in \cite{Pe}]\label{pinondegenerateifandonlyifH=eH}
A representation $(H,\p)$ of an unital $C^*$-algebra $\A$ is non-degenerate if and only if $\p(e)=I$, where $e$ is the identity of $\A$ and $I$ is the identity of $B(H)$.
\efc

\begin{defin}
Let $IHS_\A$ be the theory of Hilbert spaces together with the following conditions:
    \begin{enumerate}
    \item For $v\in Ball_1(H)$ and $a$, $b\in Ball_1(\A)$:
\[\dot{(ab)}v=(\dot{a} \dot{b})v = \dot{a}(\dot{b} v)\]
    \item For $v\in Ball_1(H)$ and $a$, $b\in Ball_1(\A)$:
\[\dot{(\half[a+b])}(v)=\half[\dot{a}+\dot{b}](v) = \half[\dot{a} v+\dot{b} v]\]
	\item\label{ItemaactingOnAverageVectors} For $v$, $w\in Ball_1(H)$, and $a\in Ball_1(\A)$:
\[\dot{a}\bigl(\half[v+w]\bigr)=\half[\dot{a} v+\dot{a} w]\]
    \item For $v\in Ball_1(H)$ and $a\in Ball_1(\A)$:
\[\langle \dot{a} v\ |\ w \rangle = \langle v\ |\dot{a^*}w \rangle\]
    \item\label{ItemNormOfa} 
		\ben
			\item For $a\in Ball_1(\A)$:
					\[\sup_v(\|\dot{a} v\|\dotminus \|a\|\|v\|)=0\]
			\item For $a\in Ball_1(\A)$:
					\[\inf_{v}\max(|\|v\|-1|,|\|\dot{a} v\|-\|a\|\ |)=0\]
		\een
    \item\label{ItemNonDegenerate} For $v\in Ball_1(H)$ and $e$ the identity element in $\A$:
\[ (\dot{ie})v=iv  \]
\[\dot{e}v=v\]
\end{enumerate}
\end{defin}

\begin{rem}
By Fact \ref{pinondegenerateifandonlyifH=eH}, Item (\ref{ItemNonDegenerate}) implies that the representation is non-degenerate. Therefore, $IHS_\A$ is the theory of the non-degenerate representations of a fixed $C^*$-algebra $\A$. Conditions in Item (\ref{ItemNormOfa}), are natural continuous logic conditions that say that $\|\p(a)\|=\|a\|$.
\end{rem}

\brm
Since the rationals of the form $\frac{k}{2^n}$ are dense in $\R$, Item (\ref{ItemaactingOnAverageVectors}) and Item (\ref{ItemNonDegenerate}) are enough to show that for all $v\in Ball_1(H)$ and all $a\in Ball_1(\A)$, we have that $(\dot{\l a})v=\l(\dot{a}v)$.
\erm

\blm\label{LemmaNonCompactOperator}
If $S:H\to H$ is a bounded operator, $S$ is non-compact if and only if for some $\l_S>0$, $S(Ball_1(H))$ contains an isometric copy of the ball of radius $\l_S$ of $\ell^2$ i.e., there exists an orthonormal sequence $(w_i)_{i\in\N}\subseteq S(Ball_1(H))$ and a vector sequence $(u_i)_{i\in\N}\subseteq Ball_1(H)$ such that for every $i\in\N$, $Su_i=\l_S w_i$.
\elm
\bpf
Suppose $S$ is non-compact. Then there is a sequence $(u^\prime_i)_{i\in\N}\subseteq Ball_1(H)$ such that no subsequence of $(Su^\prime_i)_{i\in\N}$ is convergent. By the Grahm-Schmidt process we can  assume that $(Su^\prime_i)_{i\in\N}$ is an orthogonal sequence. Since no subsequence of $(Su^\prime_i)_{i\in\N}$ converges, we have that $\liminf\{\|Su^\prime_i\|\ |\ i\in\N\}>0$ (otherwise there would be a subsequence of $Su^\prime_i$ converging to $0$). Let $\l_S:=\frac{\liminf\{\|Su^\prime_i\|\ |\ i\in\N\}}{2}>0$. For $i\in\N$, let $u_i:=\frac{\l_S u_i^\prime}{\|Su_i^\prime\|}$ and $w_i:=\frac{Su_i^\prime}{\|Su_i^\prime\|}$. Without loss of generality, we can asume that for all $i\in\N$, $\|Su_i^\prime\|>\l_S$ and therefore $\|u_i\|\leq 1$. Then, $Su_i=S(\frac{\l_Su_i}{\|Su_i\|})=\l_S\frac{Su_i^\prime}{\|Su_i^\prime\|}=\l_S w_i$.

On the other hand, suppose there are $\l_S>0$, an orthonormal sequence $(w_i)_{i\in\N}\subseteq S(Ball_1(H))$ and a vector sequence $(u_i)_{i\in\N}\subseteq Ball_1(H)$ such that for every $i\in\N$, $Su_i=\l_S w_i$. Then no subsequence of $(Su_i)_{i\in\N}$ converges and $S$ is non-compact.
\epf

\brm
If in Lemma \ref{LemmaNonCompactOperator} $\|S\|\leq 1$, it is clear that $\l_S\leq 1$.
\erm

\begin{lemma}\label{SinAe}
Let $a\in Ball_1(\A)$ be such that $\p(a)$ is a non-compact operator on $H$. Let $\l_{\p(a)}$, $(u_i)_{i\in\N}$ and $(w_i)_{i\in\N}$ be as described in Lemma \ref{LemmaNonCompactOperator}. Then, for every $n\in\N$
\begin{multline}\label{EssentialEquation}
(H,\p)\models\inf_{u_1, u_2 \cdots u_n} \inf_{w_1, w_2 \cdots w_n}\max_{i,j=1,\cdots,n}(|\langle w_i\ |\ w_j\rangle-\d_{ij}|,|au_i-\l_{\p(a)}w_i| \bigr)=0
\end{multline}
\end{lemma}
\begin{proof}
This condition is a continuous logic condition for:
\begin{multline}
\exists u_1 u_2 \cdots u_n \exists w_1 w_2 \cdots w_n\wedge\bigl(\bigwedge_{i,j=1,\cdots,n}\langle w_i\ |\ w_j\rangle=\d_{ij}\bigr)\wedge\\
\wedge\bigl(\bigwedge_{i=1,\cdots,n}\dot{a}u_i=\l_{\p(a)}w_i\bigr)
\end{multline}
where $\d_{ij}$ is Kronecker's delta. By Lemma \ref{LemmaNonCompactOperator}, this set of conditions says that $\p(a)(Ball_1(H))$ contains an isometric copy the ball of radius $\l_{\p(a)}$ of $\ell^2$.
\end{proof}

\brm\label{RemarkOperatorRankN}
It is an easy consecuence of Riesz representation theorem that if $S:H\to H$ is an operator with rank $n$, then there exist two orthonormal families $E_1:=\{u_1,\cdots,v_n\}$, $E_2:=\{w_1,\cdots,w_n\}$ and a family $\{\a_i,\dots,\a_n\}$ of non-zero complex numbers such that for every $v\in H$, $Sv=\sum_{i=1}^n\a_i\langle v\ |\ u_i\rangle w_i$. Furthermore, if $R$ is a compact operator, there is a complex sequence $(\a_i)_{i\in \N^+}$ that converges to $0$ such that for every $v\in H$, $Rv=\sum_{i=1}^\infty\a_i\langle v\ |\ u_i\rangle w_i$. If $\|R\|\leq 1$, then for every $i$, $|\a_i|\leq 1$.
\erm

\begin{lemma}\label{SinAd}
Let $n\in\N$ and $a\in Ball_1(\A)$ be such that rank$(\p(a))= n$. Let $\{\a_i,\dots,\a_n\}$ complex numbers as described in \ref{RemarkOperatorRankN}. Then
\begin{multline}\label{DiscreteEquation}
(H,\p)\models\inf_{u_1 u_2 \cdots u_n}\inf_{w_1 w_2 \cdots w_n} \sup_v \max_{i,j=1\cdots n}\bigl(|\langle u_i\ |\ u_j\rangle-\d_{ij}|,|\langle w_i\ |\ w_j\rangle-\d_{ij}|,\\
,\|\dot{a} v-\sum_{k=1}^n\a_i\langle v\ |\ u_i\rangle w_i)\|\bigr)=0
\end{multline}
\end{lemma}
\begin{proof}
This condition is a continuous logic condition for:
\begin{multline}
\exists u_1 u_2 \cdots u_n\exists w_1 w_2 \cdots w_n \bigl(\bigwedge_{i,j=1,\cdots,n}\langle u_i\ |\ u_j\rangle=\d_{ij}\wedge\langle w_i\ |\ w_j\rangle=\d_{ij}\bigr)\wedge\\
\wedge\forall v (\dot{a} v=\sum_{k=1}^n\a_i\langle v\ |\ u_i\rangle w_i)
\end{multline}
where $\d_{ij}$ is Kronecker's delta.
\end{proof}

\brm
If Condition \ref{DiscreteEquation} is valid for some $a\in\A$, and some tuple $\{\a_1,\dots,\a_n\}$, by Remark \ref{RemarkOperatorRankN} it is clear that $\p(a)$ has rank $n$.
\erm

\begin{defin}
Two representations $(H_1,\p_1)$ and $(H_2,\p_2)$ are said to be \textit{unitarily equivalent} if there exists an isometry $U$ from $H_1$ to $H_2$ such that for every $a\in\A$, $U\p_1(a)U^*=\p_2(a)$.
\edf

\begin{defin}\label{ApproximatelyUnitarilyEquivalentRepresentations}
Two representations $(H_1,\p_1)$ and $(H_2,\p_2)$ are said to be \textit{approximately unitarily equivalent} if there exists a sequence of unitary operators $(U_n)_{n<\omega}$ from $H_1$ to $H_2$ such that for every $a\in\A$ $\p_2(a)=\lim_{n\to \infty} U_n\p_1(a)U_n^*$ where the limit is taken in the norm topology.
\end{defin}

\begin{theo}[Theorem II.5.8 in \cite{Da}]\label{Voiculescu}
Two nondegenerate representations $(H_1,\p_1)$ and $(H_2,\p_2)$ of a separable $C^*$-algebra on separable Hilbert spaces are approximately unitarily equivalent if and only if, for all $a\in\A$, $\text{rank}(\p_1(a))=\text{rank}(\p_2(a))$
\end{theo}

Recall that all $C^*$-algebras under consideration are unital and all representations are nondegenerate. However, in the next lemma we do not use the hypothesis that $\A$ is unital.

\begin{lemma}\label{HensonsLemma}
Let $\A$ be a separable $C^*$-algebra of operators on the separable Hilbert space $H$, and $\p_1$ and $\p_2$ two non-degenerate representations of $\A$ on $H$. Then the structures $(H,\p_1)$ and $(H,\p_2)$ are elementarily equivalent if and only if $\p_1$ and $\p_2$ are approximately unitarily equivalent.
\end{lemma}
\begin{proof}
   \begin{itemize}
	\item[$\Rightarrow$] Suppose $(H,\p_1)\equiv(H,\p_2)$. Let $a\in Ball_1(\A)$ and assume that $\text{rank}(\p_1(a))=n<\infty$ then Condition (\ref{DiscreteEquation}) will hold in  $(H,\p_1)$. By elementary equivalence, Condition (\ref{DiscreteEquation}) will hold in  $(H,\p_2)$ and therefore $\text{rank}(\p_2(a))=n$.
	In the same way, if $\text{rank}(\p_1(a))=\infty$, Condition (\ref{EssentialEquation}) will hold in the structure $(H,\p_1)$ for every $n$. By elementary equivalence, Condition (\ref{EssentialEquation}) will hold for every $a\in\A$ and every $n$ in the structure $(H,\p_2)$ and $\text{rank}(\p_2(a))=\infty$.
	This implies that the hypotesis of Theorem \ref{Voiculescu} hold, and therefore $\p_1$ and $\p_2$ are approximately unitarily equivalent.
	\item[$\Leftarrow$] Suppose $\p_1$ and $\p_2$ are approximately unitarily equivalent. Then, there exists a sequence of unitary operators $(U_n)_{n<\omega}$ such that for every $a\in\A$, $\p_2(a)=\lim_{n\to\infty} U_n\p_1(a)U_n^*$. Let $\mathcal{F}$ be a non-principal ultrafilter over $\N$. Then $\Pi_{\FF}(H,U_n\p_1(A)U_n^*)=\Pi_{\FF}(H,\p_2)$. On the other hand, since for every $n$, $(H,U_n\p_1(A)U_n^*)\simeq (H,\p_1)$, then $\Pi_{\FF}(H,U_n\p_1(A)U_n^*)\simeq \Pi_{\FF}(H,\p_1)$. So, $\Pi_{\FF}(H,\p_1)\simeq \Pi_{\FF}(H,\p_2)$ and therefore $(H,\p_1)\equiv(H,\p_2)$.
   \end{itemize}
\end{proof}

\bth\label{equiunel}
Let $\A$ be a $C^*$-algebra, $H_1$ and $H_2$ be Hilbert spaces, and $\p_1$ and $\p_2$ be two representations of $\A$ on $H_1$ and $H_2$ respectively. Then the structures $(H_1,\p_1)$ and $(H_2,\p_2)$ are elementarily equivalent if and only if for all $a\in\A$, $\text{rank}(\p_1(a))=\text{rank}(\p_2(a))$.
\eth
\bpf
\bdc
\item[$\Rightarrow$] Suppose $(H_1,\p_1)$ and $(H_2,\p_2)$ are elementarily equivalent and let $a\in\A$. By Theorem \ref{SinAe} and Theorem \ref{SinAd}, $\text{rank}(\p(a))=n$ and $\text{rank}(\p(a))\geq n$ are sets of conditions in $L(\A)$. By elementary equivalence, $\text{rank}(\p_1(a))=\text{rank}(\p_2(a))$.
\item[$\Leftarrow$] Let $(H_1,\p_1)$ and $(H_2,\p_2)$ be such that $\text{rank}(\p_1(a))=\text{rank}(\p_2(a))$, and let $\f(a_1,\cdots,a_n)=0$ be a condition in $L(\A)$. Let $\hat{\A}\subseteq \A$ be the unital sub $C^*$-algebra of $\A$ generated by $\ba=(a_1,\cdots,a_n)$, and $\hat{\p}_1$ and $\hat{\p}_2$ be the restrictions of $\p_1$ and $\p_2$ to $\hat{\A}$ (note that $\hat{\p}_1(e)=I=\hat{\p}_1(e)$). Then $\hat{\A}$ is separable and by L\" owenheim-Skolem Theorem and Fact \ref{pinondegenerateifandonlyifH=eH},  there are two separable non-degenerate representations $(\tilde{H}_1,\tilde{\p}_1)$ and $(\tilde{H}_2,\tilde{\p}_2)$ of $\hat{\A}$ which are elementary substructures of $(H_1,\hat{\p}_1)$ and $(H_2,\hat{\p}_2)$ respectively. By Theorem \ref{Voiculescu} $(\tilde{H}_1,\tilde{\p}_1)$ is approximately unitarily equivalent to $(\tilde{H}_2,\tilde{\p}_2)$. By the previous lemma, $(\tilde{H}_1,\tilde{\p}_1)$ and $(\tilde{H}_2,\tilde{\p}_2)$ are elementary equivalent.

Then, $(\hat{H}_1,\hat{\p}_1)\models \f(a_1,\cdots,a_n)=0$ if and only if $(\hat{H}_2,\hat{\p}_2)\models \f(a_1,\cdots,a_n)=0$. But $(H,\p_1)\models \f(a_1,\cdots,a_n)=0$ if and only if $(\hat{H}_1,\hat{\p}_1)\models \f(a_1,\cdots,a_n)=0$ and $(H,\p_2)\models \f(a_1,\cdots,a_n)=0$ if and only if $(\hat{H}_2,\hat{\p}_2)\models \f(a_1,\cdots,a_n)=0$. Then $(H,\p_1)\models \f(a_1,\cdots,a_n)=0$ if and only if $(H,\p_2)\models \f(a_1,\cdots,a_n)=0$.
\edc
\epf

\bfc[Lemma I.10.7 in \cite{Da}]\label{compactrepresentations1}
Let $\A$ be an algebra of compact operators on a Hilbert space $H$. Every non-degenerate representation of $\A$ is a direct sum of irreducible representations which are unitarily equivalent to subrepresentations of the identity representation.
\efc

\bnt
For a Hilbert space $H$ and a positive integer $n$, $H^{(n)}$ denotes the direct sum of $n$ copies of $H$. If $S\in B(H)$,  $S^{(n)}$ denotes the operator on $H^{(n)}$ given by $S^{n}(v_1,\cdots,v_n)=(Sv_1,\cdots,Sv_n)$. If $\mathcal{B}\subseteq B(H)$, $\mathcal{B}^{(n)}$ is the set $\{S^{(n)}\ |\ S\in\mathcal{B}\}$.
\ent

\bdf
A representation $(H,\p)$ of $\A$ is called \textit{compact} if $\p(\A)\subseteq \mathcal{K}(H)$, where $\mathcal{K}(H)$ is the algebra of compact operators on $H$.
\edf

\bth[Theorem I.10.8 in \cite{Da}]\label{compactrepresentations2}
Let $(H,\p)$ be a compact representation of $\A$. Then for every $i\in\Z^+$, there are Hilbert spaces $H_i$ and positive integers $n_i$ and $k_i$ such that dim$(H_i)=n_i$ and
\[H\simeq \text{ker}(\p(\A))\oplus\bigoplus_{i\in \Z^+}H_i^{(k_i)}\]
and
\[\p(\A)\simeq 0\oplus\bigoplus_{i\in \Z^+}\mathcal{K}(H_i)^{(k_i)}\]
\eth

\brm
In case that $ker(\A)=0$, ($\A$ no necessarilly unital) we have that this representation is non-degenerate.
\erm

\brm
Recall that if $R\in\bigoplus_{i\in \Z^+} \mathcal{K}(H_i)^{(k_i)}$, then there is a sequence $(R_i)_{i\in \Z^+}$ such that $R_i\in\mathcal{K}(H_i)^{(k_i)}$ and $R=\sum_{i\in \Z^+}R_i$ in the norm topology. This means, in particular, that $\lim_{i\to\infty}\|R_i\|=0$.
\erm

\bdf\label{discreteparts}
Let $(H,\p)$ be a representation of $\A$. We define:
\bdc
\item[The \textit{essential part} of $\p$] It is the $C^*$-algebra homomorphism, \[\p_e:=\r\circ\p:\A\to B(H)/\mathcal{K}(H)\] of $\p(\A)$, where $\r$ is the canonical proyection of $B(H)$ onto the Calkin Algebra $B(H)/\mathcal{K}(H)$.
\item[The \textit{discrete part} of $\p$] It is the restriction,
\begin{align*}
\p_d:ker(\p_e)&\to \mathcal{K}(H)\\
a &\to \p(a)
\end{align*}
\item[The \textit{discrete part} of $\p(\A)$] It is defined in the following way:
\[\p(\A)_d:=\p(\A)\cap \mathcal{K}(H).\]
\item[The \textit{essential part} of $\p(\A)$] It is the image $\p(\A)_e$ of $\p(\A)$ in the Calkin Algebra.
\item[The \textit{essential part} of $H$] It is defined in the following way:
\[H_e:=\text{ker}(\p(\A)_d)\]
\item[The \textit{discrete part} of $H$] It is defined in the following way:
\[H_d:=\text{ker}(\p(\A)_d)^\perp\]
\item[The \textit{essential part} of a vector $v\in H$] It is the projection $v_e$ of $v$ over $H_e$.
\item[The \textit{discrete part} of a vector $v\in H$] It is the projection $v_d$ of $v$ over $H_d$.
\item[The \textit{essential part} of a set $E\subseteq H$] It is the set
\[E_e:=\{v_e\ |\ v\in E\}\]
\item[The \textit{discrete part} of a set $G\subseteq H$] It is the set
\[E_d:=\{v_d\ |\ v\in G\}\]
\edc
\edf

\blm\label{equivalentcompactrepresentations}
Let $(H_1,\p_1)$ and $(H_2,\p_2)$ be two non-degenerate representations of $\A$. If $(H_1,\p_1)\equiv(H_2,\p_2)$ then $\bigl((H_1)_d,(\p_1)_d\bigr)\simeq\bigl((H_2)_d,(\p_2)_d\bigr)$.
\elm
\bpf
For a given representation $\p$, let $\p(\A)_f$ be the (not necessarilly closed) algebra of finite rank operators in $\p(\A)$. If $(H_1,\p_1)\equiv(H_2,\p_2)$, by Lemma \ref{SinAd}, $\p_1(\A)_f\simeq\p_2(\A)_f$ and by density of $\p(\A)_f$ in $\p(\A)_d$, we have that $\p_1(\A)_d\simeq\p_2(\A)_d$. Let $\mathcal{B}:= \p_1(\A)_d\simeq\p_2(\A)_d$. Since $(H_1)_d$ and $(H_2)_d$ are the orthogonal complements of $ker(\mathcal{B})$ in $H_1$ and $H_2$ respectively, we get that $\bigl((H_1)_d,(\p_1)_d\bigr)$ and $\bigl((H_2)_d,(\p_2)_d\bigr)$ are non-degenerate representations of $\mathcal{B}$. Then by Fact \ref{compactrepresentations1} and Theorem \ref{compactrepresentations2}, $\bigl((H_1)_d,(\p_1)_d\bigr)\simeq\bigl((H_2)_d,(\p_2)_d\bigr)$.
\epf

\brm
For $E\subseteq H$, $(H_E)_e=H_{E_e}$ and $(H_E)_d=H_{E_d}$
\erm

\begin{defin}
Let $IHS_{\A,\p}$ be the theory $IHS_\A$ with the following aditional conditions:
    \begin{enumerate}
	\item For $a\in Ball_1(\A)$ such that $\p(a)$ is a non-compact operator on $H$, let $\l_{\p(a)}$, $(u_i)_{i\in\N}$ and $(w_i)_{i\in\N}$ be as described in Lemma \ref{LemmaNonCompactOperator}. For $n\in\N$
\begin{multline*}\label{EssentialEquation}
\inf_{u_1, u_2 \cdots u_n} \inf_{w_1, w_2 \cdots w_n}\max_{i,j=1,\cdots,n}(|\langle w_i\ |\ w_j\rangle-\d_{ij}|,|au_i-\l_{\p(a)}w_i| \bigr)=0
\end{multline*}	
	\item For $a\in Ball_1(\A)$, such that rank$(\p(a))= n\in\N$. Let $\a_1,\cdots,\a_n$ be complex number as described in Remark \ref{RemarkOperatorRankN}.
\begin{multline*}
\exists u_1 u_2 \cdots u_n\exists w_1 w_2 \cdots w_n \bigl(\bigwedge_{i,j=1,\cdots,n}\langle w_i\ |\ w_j\rangle=\d_{ij}\bigr)\wedge\\
\wedge\forall v (\dot{a} v=\sum_{k=1}^n\a_i\langle v\ |\ u_i\rangle w_i)
\end{multline*}
\end{enumerate}
\end{defin}

\brm
We gave in Lemmas \ref{SinAe} and \ref{SinAd} the complete continuous logic formalism only for these two conditions. We omit an explicit condition describing compact infinite rank operators in $\p(\A)$ because they completely determined by the finite rank operators in $\p(\A)$.
\erm

\begin{coro}\label{Htheory}
$IHS_{\A,\p}$ axiomatizes the theory $Th(H,\p)$.
\end{coro}
\begin{proof}
By Theorem \ref{equiunel}.
\end{proof}

\brm
Since every model of $IHS_\A$ is a non-degenerate representation $(H,\p)$, previous corollary shows that the completions of $IHS_\A$ of the form $IHS_{\A,\p}$ for some $\p$. 
\erm

\section{the models of $IHS_{\A,\p}$}\label{SectionModelsOfTpi}

In this section we provide an explicit description of the models of the theory $IHS_{\A,\p}$. This description can be summarized by stating that every model of $IHS_{\A,\p}$ can be decomposed into a fixed part (the \textit{discrete part}) and a variable part (the \textit{essential part}). Finally we get an explicit description of the monster model of $IHS_{\A,\p}$ (Theorem \ref{MonsterModel}).

\begin{defin}
Let $\A^\prime$ be the dual space of $\A$. An element $\f\in\A^\prime$ is called \textit{positive} if $\f(a)\geq 0$ whenever $a\in\A$ is positive, i.e. there is $b\in\A$ such that $a=b^*b$. The set of positive functionals is denoted by $\A^\prime_+$.
\end{defin}

\begin{lemma}\label{vectorsdefinepositivelinearfunctionals}
Let $\A$ be a $C^*$-algebra of operators on a Hilbert space $H$, and let $v\in H$. Then the function $\f_v$ on $\A$ such that for every $S\in\A$, $\f_v(S)=\langle Sv\,|\,v\rangle$ is a positive linear functional.
\end{lemma}
\begin{proof}
Linearity is clear. Let $S$ be a positive selfadjoint operator in $\A$, let $Q$ be its square root, that is, an operator such that $S=QQ^*$. Let $v\in H$; then $\langle Sv\,|\,v\rangle=\langle Q^*Qv\,|\,v\rangle=\langle Qv\,|\,Qv\rangle\geq 0$
\end{proof}

\bdf
Let $\f$ be a positive linear functional on $\A$. Let
\[\L^2(\A,\f)=\{a\in \A\ |\ \f(a^*a)<\infty\}/\sim_\f,\]
where $a_1\sim_\f a_2$ if $\f(a_1^*a_2)=0$.
For $(a)_{\sim_\f}$, $(b)_{\sim_\f}\in \L^2(\A,\f)$, let
\[\langle (a)_{\sim_\f}\ |\ (b)_{\sim_\f}\rangle_\f=\f(a^*b).\]
\edf

\bdf
We define the space $L^2(\A,\f)$ to be the completion of $\L^2(\A,\f)$ under the norm defined by $\langle\cdot\ |\ \cdot\rangle_\f$.
\edf

\brm
The product $\langle\cdot\ |\ \cdot\rangle_\f$ is a natural inner product on the space $\L^2(\A,\f)$(see \cite{Co} page 472).
\erm

\bdf
Let $\f$ be a positive linear functional on $\A$. We define the representation $M_\f:\A\to B(L^2(\A,\f))$ in the following way: For every $a\in\A$ and $(b)_{\sim_\f}\in L^2(\A,\f)$, let $M_\f(a)((b)_{\sim_\f})=(ab)_{\sim_\f}$.
\edf

\begin{defin}
$(H,\p)$ is called \textit{cyclic} if there exists a vector $v_\p$ such that $\p(\A)v_\p$ is dense in $H$. Such a vector is called a \textit{cyclic vector} for the representation $(H,\p)$. If $v_\p$ is a cyclic vector for $(H,\p)$, we denote it by $(H,\p,v_\p)$.
\end{defin}

\begin{theo}[Theorem 3.3.3. and Remark 3.4.1. in \cite{Pe}]\label{GelfandNaimarkSegal}
Let $\f$ be a positive functional on $\A$. Then there exists a cyclic representation $(H_\f,\p_\f,v_\f)$ such that for all $a\in\A$, $\f(a)=\langle \p_\f(a)v_\f|v_\f\rangle$. This representation is called the \textit{Gelfand-Naimark-Segal construction}.
\end{theo}
\bpf
Take $(L^2(\A,\f_v),M_{\f_v},(e)_{\sim_{\f_v}})$. Note that
\[\langle M_{\f_v}(a)(e)_{\sim_{\f_v}}\ |\ (e)_{\sim_{\f_v}}\rangle=\langle (a)_{\sim_{\f_v}}\ |\ (e)_{\sim_{\f_v}}\rangle=\f_v(a\cdot e)=\f_v(a).\]
\epf

\bdf
 Two cyclic representations $(H_1,\p_1,v_1)$ and $(H_2,\p_2,v_2)$ are said to be \textit{isometrically isomorphic} if there is an isometry $U$ from $H_1$ to $H_2$ such that for every $a\in\A$, $U\p_1(a)U^*=\p_2(a)$ and $Uv_1=v_2$.
\end{defin}

\begin{theo}[Proposition 3.3.7 in \cite{Pe}]\label{UnitarilyEquivalentRepresentations}
Two cyclic representations $(H_1,\p_1,v_1)$ and $(H_2,\p_2,v_2)$ are isometrically isomorphic if and only if for all $a\in\A$, $\langle \p_1(a)v_1|v_1\rangle=\langle \p_2(a)v_2|v_2\rangle$.
\end{theo}

\begin{theo}\label{HvL2}
Let $v\in H$. Then $(H_v,\p_v,v)\simeq (L^2(\A,\f_v),M_{\f_v},(e)_{\sim_{\f_v}})$.
\end{theo}
\bpf
By Gelfand-Naimark-Segal Theorem \ref{GelfandNaimarkSegal} and Theorem \ref{UnitarilyEquivalentRepresentations}.
\epf

\bth[Remark 3.3.1. in \cite{Pe}]\label{RemarkEveryRepresentationIsADirectSumOfCyclicRepresentations}
Every representation can be seen as a direct sum of cyclic representations.
\eth

\begin{defin}\label{definStates}
We define the following (see \cite{Pe}):
\ben
\item A positive linear functional $\f$ on $\A$ is called a \textit{quasistate} if $\|\f\|\leq 1$.
\item The set of the of quasistates on $\A$ is denoted by $Q_\A$.
\item In the case where $\|\f\|=1$, the positive linear functional $\f$ is called a \textit{state}.
\item The set of states is denoted by $S_\A$.
\item A state is called \textit{pure} if it is not a convex combination of other states.
\item The set of pure states is denoted by $PS_\A$.
\een
\end{defin}

\begin{defin}
Let $(H_i,\p_i)$ for $i\in I$ be a family of representations of $\A$. We define a representation $\oplus\p_i$ on $\oplus H_i$ in the following way:
Let $v=\sum_iv_i$ and $a\in\A$, $\oplus\p_i(a)v=\sum_i\p_i(a)v_i$.
\end{defin}

\begin{defin}\label{HA}
Let $H_{S_\A}$ be the space,
\[H_{S_\A}=\oplus_{\f\in S_\A} L^2(\A,\f)\]
and let $\p_{S_\A}$ be,
\[\p_{S_\A}=\oplus_{\f\in S_\A} M_\f,\]
\end{defin}

\begin{defin}\label{DefinitionHv}
Given $E\subseteq H$ and $v\in H$, we denote by:
\ben
\item $H_E$, the Hilbert subspace of $H$ generated by the elements $\p(a)v$, where $v\in E$ and $a\in\A$.
\item $\p_E:=\{\p(a)\upharpoonright H_E\ |\ a\in \A\}$.
\item $(H_E,\p_E)$, the subrepresentation of $(H,\p)$ generated by $E$.
\item $H_v$, the space $H_E$ when $E=\{v\}$ for some vector $v\in H$
\item $\p_v:=\p_E$ when $E=\{v\}$.
\item $(H_v,\p_v)$, the subrepresentation of $(H,\p)$ generated by $v$.
\item $H_E^\perp$, the orthogonal complement of $H_E$
\item $P_E$, the projection over $H_E$.
\item $P_{E^\perp}$, the projection over $H_E^\perp$.
\een
\end{defin}

\begin{rem}\label{CyclicVector}
For $v\in H$, it is clear that $v$ is a cyclic vector for $\A$ on $H_v$.
\end{rem}

\brm\label{RemarkProjectionsOnTuples}
For a tuple $\bv=(v_1,\dots,v_n)$, by $P_E\bv$ we denote the tuple $(P_Ev_1,\dots,P_Ev_n)$.
\erm

\bfc\label{Hdsubsetalgebraicclosureofemptyset}
Let $v\in H_d$. Then $v$ is algebraic over $\emptyset$.
\efc
\bpf
If $v\in H_d$ by Theorem \ref{compactrepresentations2}, there exist a sequence $v_i$ of vectors in $H_d$ such that $v_i\in H_i^{k_i}$, and $v=\sum_{i\geq 1}v_k$. Given that $\|v_k\|\to 0$ when $k\to\infty$, the orbit of $v$ under any automorphism $U$ of $(H,\p)$ is a Hilbert cube which is compact, which implies that $v$ is algebraic.
\epf

\bfc[Proposition 2.7 in \cite{HeTe}]\label{LemmaHensonTellez}
Let $\MM$ and $\mathcal{N}$ be $\mathcal{L}$-structures, $A\subseteq M$ and $B\subseteq N$. If $f:A\to B$ is an elementary map, then there is an elementary map $g:acl_\MM(A)\to acl_{\mathcal{N}}(B)$ extending $f$. Moreover, if $f$ is onto, then so is $g$.
\efc

\bth\label{equivalentrepresentations}
Let $(H_1,\p_1)$ and $(H_2,\p_2)$ be two representations of $\A$. Then $(H_1,\p_1)\equiv(H_2,\p_2)$ if and only if
\[\bigl((H_1)_d,(\p_1)_d\bigr)\simeq\bigl((H_2)_d,(\p_2)_d\bigr)\]
 and
 \[((H_1)_e,(\p_1)_e)\equiv((H_2)_e,(\p_2)_e)\]
\eth
\bpf
By Theorem \ref{equiunel}, $(H_1,\p_1)\equiv(H_2,\p_2)$ if and only if $((H_1)_d, (\p_1)_d)\equiv((H_2)_d,(\p_2)_d)$ and $((H_1)_e,(\p_1)_e)\equiv((H_2)_e,(\p_2)_e)$. By Lemma \ref{equivalentcompactrepresentations}, this is equivalent to $\bigl((H_1)_d,(\p_1)_d\bigr)\simeq\bigl((H_2)_d,(\p_2)_d\bigr)$ and $((H_1)_e,(\p_1)_e)\equiv((H_2)_e,(\p_2)_e)$.
\epf

\begin{theo}\label{MonsterModel}
Let $\k\geq|S_\A|$ be such that cf$(\k)=\k$. Then the structure \[(\tilde{H}_\k,\tilde{\p}_\k)= (H_d,\p_d)\oplus\bigoplus_\k(H_{S_{\p(\A)_e}},\p_{S_{\p(\A)_e}})\] is $\k$ universal, $\k$ homogeneous and is a monster model for $Th(H,\p)$.
\end{theo}
\begin{proof}
Let us denote $(\tilde{H}_\k,\tilde{\p}_\k)$ just by $(\tilde{H},\tilde{\p})$.
\bdc
\item[$(\tilde{H},\tilde{\p})\models Th(H,\p)$] For every $a\in Ball_1(\A)$, if $rank(a)=\infty$ in $H_e$, then $rank(a)=\infty$ in $\tilde{H}_e$ and if $rank(a)=0$ in $H_e$, then $rank(a)=0$ in $\tilde{H}_e$. By Theorem \ref{equivalentrepresentations} $(H_e,\p_e)\equiv (\tilde{H}_e,\tilde{\p}_e)$. By Theorem \ref{equivalentrepresentations}, $(H,\p)\equiv(\tilde{H},\tilde{\p})$.
\item[$\k$-Universality] Let $(H^\prime,\p^\prime)\models Th(H,\p)$ be a model with density less than $\k$. By Theorem \ref{equivalentrepresentations}, $(H^\prime,\p_d^\prime)\simeq (\tilde{H}_d,\tilde{\p}_d)\simeq (H_d,\p_d)$. Then without loss of generality we can asume that $\p(\A)=\p(\A)_e$. By Theorem \ref{RemarkEveryRepresentationIsADirectSumOfCyclicRepresentations}, there exists a set $I$ and a family $(H_i,\p_i,v_i)_{i\in I}$ of cyclic representations such that $(H^\prime,\p^\prime)=\bigoplus_{i\in I}(H_i,\p_i)$. By Theorem \ref{HvL2}, $(H_{v_i},\p_{v_i},v_i)\simeq (L^2(\A,\f_{v_i}),M_{\f_{v_i}},(e)_{\sim_{\f_{v_i}}})$. Since the density of $(H^\prime,\p^\prime)$ is less than $\k$, the size of $I$ is less than $\k$ and clearly $(H^\prime,\p^\prime)$ is isomorphic to a subrepresentation of $(\tilde{H},\tilde{\p})$.

\item[$\k$-Homogeneity] Let $U$ be a partial elementary map between $E$, $F\subseteq \tilde{H}$ with $|E|=|F|<\k$.
\ben
\item We can extend $U$ to an unitary equivalence between $H_E$ and $H_F$: Let $a_1$, $a_2\in\A$ and $e_1$, $e_2\in E$. Then we define $U(\p(a_1)(e_1)+\p(a_2)(e_2)):=\p(a_1)(U(e_1))+\p(a_2)(U(e_2))$. After this, we extend this constuction continuously to $H_E$.
\item We can extend $U$ to an unitary equivalence between $(H_d\oplus H_{E_e})$ and $(H_d\oplus H_{F_e})$: By Lemma \ref{Hdsubsetalgebraicclosureofemptyset}, $(H_d\oplus H_{E_e})\subseteq acl_{\tilde{H}}(E)$ and $(H_d\oplus H_{F_e})\subseteq acl_{\tilde{H}}(F)$. By Fact \ref{LemmaHensonTellez}, we can extend $U$ in the desired way.
\item We can find an unitary equivalence between $(H_d\oplus H_{E_e})^\perp$ and $(H_d\oplus H_{F_e})^\perp$: Given that $|E|=|F|<\k$, there are two subsets $C_1$ and $C_2$ of $\k$ such that $(H_d\oplus H_{E_e})^\perp = \bigoplus_{C_1}H_{PS_{\p(\A)_e}}$ and $(H_d\oplus H_{E_e})^\perp=\bigoplus_{C_2}H_{PS_{\p(\A)_e}}$. We have that $|C_1|=|C_2|=\k$ and therefore,
\[\bigoplus_{C_1}(H_{S_{\p(\A)_e}},\p_{S_{\p(\A)_e}})\simeq\bigoplus_{C_2}(H_{S_{\p(\A)_e}},\p_{S_{\p(\A)_e}}).\]
 Let $U^\prime$ an isomorphism between
\[\bigoplus_{C_1}(H_{S_{\p(\A)_e}},\p_{S_{\p(\A)_e}})\text{ and }\bigoplus_{C_2}(H_{S_{\A_e}},\p_{S_{\p(\A)_e}}).\]
\item Let $v\in\tilde{H}_\k$. Then $v=v_d+v_{E_e}+v_{E_e^\perp}$, where $v_{E_e}:=P_{E_e}v$ and $v_{E_e^\perp}:=P_{E_e^\perp}v$. Let $w:=Uv_d+Uv_{E_e}+U^\prime v_{E_e^\perp}$, and $U^{\prime\prime}:=U\oplus U^\prime$. Then $w$ and $U^{\prime\prime}$ are such that $U^{\prime\prime}$ is an automorphism of $\tilde{H}_\k$ extending $U$ such that $U^{\prime\prime}v=w$.
\een
\edc
\end{proof}

\section{types and quantifier elimination}\label{sectiontypesquantifierelimination}
In this section we provide a characterization of types in $(H,\p)$. The main results here are Theorem \ref{typeoverempty1} and Theorem \ref{typeoverA} that characterize types in terms of subrepresentations of $(H,\p)$ and, its consecuence, Corollary \ref{QuantifierElimination} that states that $IHS_{\A,\p}$ has quantifier elimination. As in the previous section, we denote by $(\tilde{H},\tilde{\p})$ the monster model for the theory $IHS_{\A,\p}$ as constructed in Theorem \ref{MonsterModel}.

\brm\label{autounitary}
An automorphism $U$ of $(H,\p)$ is a unitary operator $U$ on $H$ such that $U\p(a)=\p(a)U$ for every $a\in Ball_1(\A)$.
\erm
\begin{proof}
Asume $U$ is an automorphism of $(H,\p)$. It is clear that $U$ must be a linear operator. Also, for every $v, w\in H$ and  $\p(a)\in\A$, we must have that $U(\p(a)v)=\p(a)(Uv)$ and $\langle Uv\,|\,Uw\rangle=\langle v\,|\, w\rangle$ by definition of automorphism. Therefore $U$ must be unitary and commutes with the elements of $\p(\A)$. Conversely, if $U$ is an unitary operator commuting with the elements of $\p(\A)$, then $U$ is clearly an automorphism of $(H,\p)$.
\end{proof}

\blm\label{LemmaComponentsOfHdInvariant}
Let
\[H_d=\bigoplus_{i\in \Z^+}H_i^{(k_i)}\]
be as in Theorem \ref{compactrepresentations2}. Let $v\in H_i^{(k_i)}$ for some $i\in \Z^+$ and let $U\in \Aut(H,\p)$. Then $Uv\in H_i^{(k_i)}$.
\elm
\bpf
By Theorem \ref{compactrepresentations2}, $\p(\A)_d=\p(\A)\cap \mathcal{K}(H)$ can be seen as:
\[\p(\A)_d=\bigoplus_{i\in \Z^+}\mathcal{K}(H_i^{(k_i)}).\]
By Remark \ref{autounitary}, any automorphism $U\in \Aut((H,\p))$ commutes with every element of $\p(\A)$, in particular with any element $K$ of $\mathcal{K}(H_i^{(k_i)})$. Thus, if $v\in H_i^{(k_i)}$ and $K\in\mathcal{K}(H_i^{(k_i)})$, $KUv=UKv$. This implies that $Uv\in H_i^{(k_i)}$.
\epf

\begin{theo}\label{typeoverempty1}
Let $v$, $w\in\tilde{H}$. Then $tp(v/\emptyset)=tp(w/\emptyset)$ if and only if $(H_v,\p_v,v)$ is isometrically isomorphic to $(H_w,\p_w,w)$.
\end{theo}
\begin{proof}
Let us suppose that $tp(v/\emptyset)=tp(w/\emptyset)$. Then there is an automorphism $U$ of $(\tilde{H},\tilde{\p})$ such that $Uv=w$. Therefore the representations $(H_v,\p_v,v)$ and $(H_w,\p_w,w)$ are unitarily equivalent and therefore $(H_v,\p_v,v)$ is isometrically isomorphic to $(H_w,\p_w,w)$.

Conversely, let $(H_v,\p_v,v)$ be isometrically isomorphic to $(H_w,\p_w,w)$. By Theorem \ref{MonsterModel}, $(H_v,\p_v)$ and $(H_w,\p_w)$ can be seen as subrepresentations of $(\tilde{H},\tilde{\p})$. Given that $(H_v,\p_v,v)$ and $(H_w,\p_w,w)$ are isometrically isomorphic, by Theorem \ref{compactrepresentations2} and Theorem \ref{MonsterModel}, the decompositions of $(H_v,\p_v)$ and $(H_w,\p_w)$ into cyclic representations are isometrically isomorphic too, and therefore $\tilde{H}_v^\perp$ and $\tilde{H}_w^\perp$ are isometrically isomorphic. Then we get an automorphism of $(\tilde{H},\tilde{\p})$ that sends $v$ to $w$, and $v$ and $w$ have the same type over the empty set.
\end{proof}

\begin{theo}\label{typeoverempty4}
Let $v,w\in H$. Then $tp(v/\emptyset)=tp(w/\emptyset)$ if and only if $\f_v=\f_w$, where $\f_v$ denotes the positive linear functional on $\A$ defined by $v$ as in Lemma \ref{vectorsdefinepositivelinearfunctionals}.
\end{theo}
\bpf
Let $v$ and $w\in H$ be such that tp$(v/\emptyset)=$tp$(w/\emptyset)$. Then qftp$(v/\emptyset)=$qftp$(w/\emptyset)$ and therefore, for every $a\in \A$, $\langle \p(a)v|v\rangle=\langle \p(a)w|w\rangle$. But this means that $\f_v=\f_w$.

Conversely, if $\f_v=\f_w$, by Theorem \ref{UnitarilyEquivalentRepresentations}, $(H_v,\p_v,v)$ is isometrically isomorphic to $(H_w,\p_w,w)$ and by Theorem \ref{typeoverempty1} tp$(v/\emptyset)=$tp$(w/\emptyset)$.
\epf

\begin{lemma}\label{Aut(H/E)}
Let $E\subseteq H$, $U\in \Aut(H,\p)$. Then $U\in \Aut((H,\p)/E)$ if and only if $U\upharpoonright (H_E,\p_E)=Id_{(H_E,\p_E)}$.
\end{lemma}
\bpf
Suppose that $U\upharpoonright (H_E,\p_E)=Id_{(H_E,\p_E)}$. Then, $U$ fixes $H_E$ pointwise, and, therefore, fixes $E$ pointwise.
Conversely, suppose $U\in \Aut((H,\p)/E)$. By Remark \ref{autounitary}, $U$ is an unitary operator that commutes with every $S\in\p(\A)$. Then for every $S\in\p(\A)$ and $v\in E$, we have that $U(Sv)=S(Uv)=Sv$. So $U$ acts on $H_E$ like the identity and the conclusion follows.
\epf

\begin{theo}\label{typeoverA}
Let $v$ and $w\in\tilde{H}$ and $E\subseteq\tilde{H}$. Then tp$(v/E)=$tp$(w/E)$ if and only if $P_E(v)=P_E(w)$ and tp$(P_E^\perp(v)/\emptyset)=$tp$(P_E^\perp(w)/\emptyset)$.
\end{theo}
\begin{proof}
\bdc
\item[$\Rightarrow$] Suppose tp$(v/E)=$tp$(w/E)$. Given that tp$(v/E)=$tp$(w/E)$, there exists $U\in\Aut((\tilde{H},\tilde{\p})/E)$ such that $Uv=w$. By Lemma \ref{Aut(H/E)},  $U\upharpoonright (H_E,\p_E)=Id_{(H_E,\p_E)}$ and $U(P_E(v))=P_E(Uv)=P_E(w)$. On the other hand, $U(P_E^\perp(v))=P_E^\perp(w)$ and therefore tp$(P_E^\perp(v)/\emptyset)=$tp$(P_E^\perp(w)/\emptyset)$.
\item[$\Leftarrow$] Asume $P_E(v)=P_E(w)$ and tp$(P_E^\perp(v)/\emptyset)=$tp$(P_E^\perp(w)/\emptyset)$. Then there exists an automoprhism $U$ of $(\tilde{H},\tilde{\p})$ such that $U(P_E^\perp(v))=P_E^\perp(w)$. Let $\tilde{U}=Id_{H_E}\oplus (U\upharpoonright\tilde{H}_E^\perp)$. Then, by Lemma \ref{Aut(H/E)}, $\tilde{U}$ is an automorphism of $(\tilde{H},\tilde{\p})$ that fixes $E$ pointwise and  $Uv=w$. This implies that tp$(v/E)=$tp$(w/E)$.
\edc
\end{proof}

\begin{coro}\label{QuantifierElimination}
The structure $(H,\p)$ has quantifier elimination.
\end{coro}
\bpf
This follows from Theorem \ref{typeoverA} that shows that types are determined by quantifier-free conditions contained in it.
\epf

\brm
Note that the quantifier elimination result that we proved is uniform. That is, types are isolated by the conditions $\langle\dot{a}x\ |\ x\rangle$ for $a\in\A$ no matter the particular non-degenerate representation $(H,\p)$ we choose.
\erm

Recall that the weak$^*$ topology in $\A^\prime$ (the Banach dual Algebra of $\A$) is the coarsest topology in $\A^\prime$ such that for every $a\in A$, the function $F_a:\A^\prime\to\C$ is continuous, where $F_a(\f)=\f(a)$ for $a\in\A$ and $\f\in\A^\prime$.

\bth
The stone space $S_1(Th(H,\p))$ (i.e. the set of types of vectors of norm less than or equal to $1$) with the logic topology is homeomorphic to the quasi state space $Q_\A$ with the weak$^*$ topology.
\eth
\bpf
We consider types of vectors with norm less than or equal to $1$. Similarly, we consider positive linear functionals with norm less than or equal to $1$, that is, the quasi state space $Q_\A$. By Theorem \ref{typeoverempty4}, types of vectors in $H$ are determined by the corresponding positive linear functionals, explicitly $v\to \f_v(a):=\langle \p(a)v\ |\ v\rangle$. So, there is a bijection between $S_1(Th(H,\p))$ and $Q_\A$. \
To prove bicontinuity, let $h:S_1(Th(H,\p))\to Q_\A$ be the previously defined bijection. Let $X$ be a weak$^*$ basic open set in $Q_\A$; then there exists an open sets $V_1,\dots,V_k\subseteq \C$ and elements $a_1,dots,a_k\in\A$ such that for every $\f\in Q_\A$, we have that $\f\in X$ if and only if $\f(a_i)\in V_i$. For $\f\in X$ let $v_\f$ be a cyclic vector such that $\f=\f_{v_\f}$. Then for every $\f\in X$ and $i=1,\dots,k$, $\langle \p(a_i)v_\f\ |\ v_\f\rangle\in V_i$ but this condition defines an open set in $S_1(Th(H,\p))$.\\ Conversely, by quantifier elimination, every basic open sets $X$ in the logic topology in $S_1(Th(H,\p))$ can be expresed as finite intersection of sets with the form:
\[\{p\in S_1(Th(H,\p))\ |\ v_\f\models p\Rightarrow\langle \p(a)v_\f\ |\ v_\f\rangle\in V\}\]
where $V\subseteq \C$ open. Each of this sets is in correspondence by $h$ with a set of the form
\[\{\f\in Q_\A\ |\ \langle \p(a)v_\f\ |\ v_\f\rangle\in V\}\]
which defines an open set in $Q_\A$.
\epf

\section{a model companion for $IHS_\A$}\label{SectionModelCompanionForTA}
In this section we prove that the theory $IHS_\A$ has a model companion (Theorem \ref{TAmodelCompanion}). This result goes in the same direction as results from Ben Yaccov and Usvyatsov (\cite{BU}) and Berenstein and Henson (\cite{BerHen}).
\bdf
Let $EIHS_\A$ be the theory of a representation $(H,\p)$ such that no element of $\A$ acts as a compact operator.
\edf

\bth\label{TheoremModelosTAentreModelosTA*}
For every Hilbert space representation $(H,\p)\models IHS_\A$, there is a Hilbert space representation $(H^\prime,\p^\prime)\models EIHS_\A$, such that $(H,\p)\subseteq (H^\prime,\p^\prime)$.
\eth
\bpf
Let $H=H_d\oplus H_e$ as in Definition \ref{discreteparts}. We define:\[H^\prime:=\bigl(\bigoplus_\omega H_d\bigr)\oplus H_e\]
and for every $a\in\A$ such that $\p(a)$ is compact in $H$,
\[\p^\prime(a):\bigl(\bigoplus_\omega \p(a)\bigr)\oplus 0\]
and
\[\p^\prime(a):=0\oplus\p(a)\]
for every $a\in\A$ such that $\p(a)$ is non-compact in $H$.

Then, $(H^\prime,\p^\prime)$ is clearly a representation of $\A$. Let $a\in\A$ such that $\p(a)$ is compact in $H$. Since the non-zero eigenvalues of $\p^\prime(a)$ have infinite dimensional eigen spaces, the operator $\p^\prime(a)$ cannot be compact in $(H^\prime,\p^\prime)$. Therefore all the elements of $\A$ that acted compactly on $H$ no longer act compactly on $H^\prime$. Since the elements of $\A$ that acted non-compactly on $H$ still act non-compactly on $H^\prime$, no element of $\A$ act compactly on $H^\prime$ and therefore, $(H^\prime,\p^\prime)\models EIHS_\A$.
\epf

\bco\label{TAmodelCompanion}
$EIHS_\A$ is a model companion for $IHS_\A$.
\eco
\bpf
Clearly, every model of $EIHS_\A$ is a model of $IHS_\A$. On the other hand, by Theorem \ref{TheoremModelosTAentreModelosTA*}, every model of $IHS_\A$ can be (non elementarily) embedded in a model of $EIHS_\A$. These previous fact show that $EIHS_\A$ is a companion for $IHS_\A$. Since by Corollary \ref{QuantifierElimination} the theory $EIHS_\A$ is model complete, the theory $EIHS_\A$ is a model companion for $IHS_\A$.
\epf

\end{document}